\documentclass{amsart}
\usepackage{amsmath,amssymb,amsthm,latexsym}
\usepackage[all]{xy} 

\newcommand{\R}{\mathbb{R}}
\newcommand{\Rn}{\R^n}
\newcommand{\tn}{\mathbb{T}^n}
\newcommand{\Z}{\mathbb{Z}}
\newcommand{\cotan}{\mathrm{T}^*}
\newcommand{\uc}{\mathbb{R}^3 \setminus \{0\}}
\newcommand{\de}{\mathrm{d}}
\newcommand{\fib}{\xymatrix@1{ F \;\ar@{^{(}->}^-{\iota}[r] & M
    \ar[r]^-{\pi} & B}}
\newcommand{\fibr}{\xymatrix@1{ F \;\ar@{^{(}->}[r]^-{\iota'} & M'
    \ar[r]^-{\pi'} & B}}
\newcommand{\fibpl}{\xymatrix@1{ \mathrm{H}_1(F,\Z) \;\ar@{^{(}->}[r] & \mathcal{P}
    \ar[r] & B}}
\newcommand{\fibrpl}{\xymatrix@1{ \mathrm{H}_1(F,\Z) \;\ar@{^{(}->}[r] & \mathcal{P}'
    \ar[r] & B}}
\newcommand{\affr}{\mathrm{Aff}(\Rn)}
\newcommand{\afft}{\mathrm{Aff}(\tn)}
\newcommand{\gln}{\mathrm{GL}(n,\Z)}
\newcommand{\zn}{\Z^n}
\newcommand{\caff}{\mathrm{B}\afft}
\newcommand{\toral}{I \times \tn}
\newcommand{\linear}{\gln \times \mathrm{0}}
\newcommand{\ctoral}{\mathrm{B}(\toral)}
\newcommand{\clinear}{\mathrm{B}(\linear)}
\newcommand{\cgln}{\mathrm{B}\gln}
\newcommand{\rpr}{\R\mathrm{P}^2 \times \R}
\newcommand{\tth}{\mathbb{T}^3}
\newcommand{\zt}{\mathbb{Z}^3}
\newcommand{\rt}{\mathbb{R}^3}
\newcommand{\ztw}{\mathbb{Z}_2}
\newcommand{\xvec}{\mathbf{x}}
\newcommand{\tvec}{\mathbf{t}}
\newcommand{\twcoho}{\mathrm{H}^2(\rpr ; \zt_{f_*})}
\newcommand{\rpt}{\mathbb{R}\mathrm{P}^2}
\newcommand{\twco}{\mathrm{H}^2(\rpt; \zt_{f_*})}

\theoremstyle{definition}
\newtheorem{defn}{Definition}[section]
\newtheorem{exm}{Example}[section]
\newtheorem{claim}{Claim}[section]
\newtheorem{rk}{Remark}[section]

\theoremstyle{plain}
\newtheorem{thm}{Theorem}[section]
\newtheorem{lemma}{Lemma}[section]
\title{Topological classification of Lagrangian fibrations}
\author{Daniele Sepe}
\address{School of Mathematics and Maxwell Institute for Mathematical Sciences, The University of Edinburgh, James
  Clerk Maxwell Building, King's Buildings, Edinburgh, EH9 3JZ, UK}
\email{d.sepe@sms.ed.ac.uk}
\thanks{I would like to thank Toby
  Bailey for his insightful comments on earlier versions of this note.}
\begin{document}
\begin{abstract}
 We define topological invariants of regular Lagrangian fibrations
 using the integral affine structure on the base space and we show
 that these coincide with the classes known in the literature. We also
 classify all symplectic types of Lagrangian fibrations with base
 $\rpr$ and fixed monodromy representation, generalising a
 construction due to Bates in \cite{bates_ob}.
\end{abstract}
\maketitle
\tableofcontents

\section{Introduction}\label{sec:introduction}
The geometry and topology of completely integrable Hamiltonian systems
have been of interest since the $19^{\mathrm{th}}$ century. By a \emph{completely
  integrable Hamiltonian system}, in this paper we mean a
$2n$-dimensional symplectic manifold $(M,\omega)$ and functions
$f_1,\ldots,f_n : M \to \R$ satisfying the following conditions

\begin{enumerate}
\item they are in \emph{involution}, meaning that $\{f_i,f_j\} = 0$
  for all $i,j$, where $\{.,.\}$ denotes the Poisson bracket induced
  by the symplectic form;
\item they are \emph{functionally independent} almost everywhere - \textit{ i.e.} $\de f_1 \wedge \ldots \wedge \de f_n \neq 0$ on a set of full measure.
\end{enumerate}
\noindent
We remark that we have only used the Poisson structure on $M$ to
define a completely integrable Hamiltonian system and so there is an
associated notion for Poisson manifolds, which we shall not consider
here.\\
\linebreak[1]
\indent
In this paper, we study the global topology of Lagrangian
fibrations, first introduced by Duistermaat in \cite{dui} in the
theory of completely integrable Hamiltonian systems.

\begin{defn}
We say a fibration $\fib$ is a \emph{Lagrangian fibration} if
$(M,\omega)$ is a $2n$-dimensional symplectic manifold, $B$ is an
$n$-dimensional connected manifold, $F$ is compact and
$\omega|_{\pi^{-1}(b)} = 0 $ for all $b \in B$.  
\end{defn}

Duistermaat in \cite{dui} showed that, locally, a Lagrangian fibration
is given by a completely integrable Hamiltonian system; hence, local
properties of the latter are still true for these fibrations. The
Liouville-Arnold theorem completely determines the geometric structure
of a completely integrable Hamiltonian system in a neighbourhood of a
regular level of the map $\mathbf{f}=(f_1,\ldots,f_n) : M \to \Rn$. For completeness, we state
this classical theorem, referring the reader to \cite{arnold} or \cite{dui} for a proof.

\begin{thm}[Liouville-Arnold Theorem] \label{thm:la}
  Let $\mathbf{f}=(f_1,\ldots,f_n): M \to \Rn$ be a completely integrable
  Hamiltonian system with $n$ degrees of freedom and let $x \in \Rn$
  be a regular value in the image of $\mathbf{f}$. Suppose that $\mathbf{f}^{-1}(x)$ has
  a compact, connected component, denoted by $F_x$. Then
  \begin{itemize}
  \item $F_x$ is a Lagrangian submanifold of $M$ and is diffeomorphic
    to $\tn$;
  \item there is a neighbourhood $V$ of $F_x$ in $M$ that is
    symplectomorphic to an open neighbourhood $W$ of $\tn$ in $\cotan
    \tn$, as shown in the diagram below.
    \begin{displaymath}
       \xymatrix{M \ar@{<-^{)}}[r] \ar@{<-^{)}}[dr] &\; V \ar[r]^-{\tilde{\varphi}} \ar[d]
& W\; \ar[d] \ar@{^{(}->}[r] & \mathrm{T}^*\tn  \\
& \; F_x \ar[r]^-{\varphi} & \tn \;\ar@{^{(}->}[ur] &}
    \end{displaymath}
  \end{itemize}
\noindent

Clearly, $W \cong D^n \times \tn$. Let $\tilde{\varphi}(p)=
(a^1(p),\ldots,a^n(p),\alpha^1(p), \ldots, \alpha^n(p))$.
The coordinates $a^i$ are called the \textbf{actions} and depend
smoothly on the functions $f_i$. The coordinates $\alpha^i$ are called
the \textbf{angles}.
\end{thm}

In \cite{dui}, Duistermaat found that there are two topological
obstructions to finding a \emph{global} set of action-angle
coordinates in a Lagrangian fibration, namely

\begin{itemize}
\item the \emph{monodromy} - which can be seen as the obstruction for
  the fibration to be a principal $\tn$-bundle (Theorem \ref{thm:la}
  guarantees that a Lagrangian fibration is locally a principal
  $\tn$-bundle);
\item the \emph{Chern class} - a generalisation of the Chern class for
  principal $\tn$-bundles, which can be seen as the obstruction for
  the fibration to admit a global section $s:B \to M$.
\end{itemize}
\noindent
More recent work by Dazord and Delzant \cite{daz_delz} (where they
deal with the more general concept of isotropic fibrations) and by
Zung in \cite{zung_symp1} and \cite{zun_symp2} has shun more light on the global
geometric and topological properties of Lagrangian fibrations. In
\cite{zun_symp2} Zung introduced a \emph{symplectic} invariant of Lagrangian
fibrations called the \emph{Lagrangian class}, which can be seen as
the obstruction to the existence of a symplectomorphism between two
Lagrangian fibrations with the same monodromy and Chern class. He also
generalised these invariants to singular Lagrangian fibrations, using
a sheaf theoretic approach to the definition of these invariants (in
line with the traditional approach). In this paper, we shall only be
concerned with regular Lagrangian fibrations (i.e. with no
singularities) and we will present an alternative approach to the
definition of monodromy and Chern class using a \emph{classifying
  space}. This method relies on the observation that Lagrangian
fibrations over a base space $B$ are intimately related to
\emph{integral affine structures} on $B$, a relation that is
illustrated by Lemma \ref{lemma:ias}. We remark that the importance of
the integral affine structure in studying Lagrangian fibrations has
been highlighted in various papers (see Bates' paper \cite{bates_monchamp} and
Zung's \cite{zun_symp2}).\\

There are various well-known examples of twisted Lagrangian fibrations
in the literature. Cushman in \cite{cush_sph} and Duistermaat in
\cite{dui} explained how monodromy arises in a neighbourhood of a
singularity of a real-world integrable system - the spherical
pendulum. Many other examples of real-world systems with monodromy are
now known (see \cite{cush_dav} for some more examples) and in \cite{zung_ff}
Zung went as far as proving that in systems with two degrees of freedom monodromy
can only arise around focus-focus singularities. Bates in \cite{bates_ob}
constructed examples of Lagrangian fibrations over $\uc \cong S^2
\times \R $ with non-trivial Chern class (and trivial monodromy, since
$\pi_1(\uc)=0$). In this paper, we explicitly construct examples of Lagrangian
fibrations with non-trivial Chern classes and non-trivial monodromy
(but trivial Lagrangian classes). We remark that, once we show that
the base space of our fibrations possesses an integral affine structure,
results of Dazord and Delzant (in \cite{daz_delz}) and of Zung (in
\cite{zun_symp2}) show that these fibrations can be constructed. In our
construction, we use methods from \cite{bates_ob} and \cite{sf}.\\

The structure of this paper is the following. In section
\ref{sec:ias}, we explain the relation between integral affine
structures and Lagrangian fibrations, presenting a well-known lemma
(Lemma \ref{lemma:ias}), which we have not found proved in the
literature and is therefore included for completeness. In section \ref{sec:cs}, we use the integral affine
structure induced by the Lagrangian fibration on the fibres to
construct the monodromy and Chern class of Lagrangian fibrations using
classifying spaces. Examples of twisted Lagrangian fibrations are
presented in section \ref{sec:exs}.


\section{Integral Affine Structures and Lagrangian fibrations}\label{sec:ias}
In this section we establish the relation between integral affine
structures on a manifold $B$ and Lagrangian fibrations with base space
$B$.
\begin{defn}
  \label{defn:iag}
The semidirect product $\affr = \gln \ltimes \Rn$ where multiplication
is defined by

\begin{equation}
  \label{eq:1}
  (A,x) \cdot (B,y) = (AB, Ay+x)
\end{equation}
is called the group of \emph{integral affine transformations}.
\end{defn}
We note that this is equivalent to saying that there is a short
exact sequence

\begin{equation}
  \label{eq:2}
  \xymatrix@1{1 \ar[r] & \Rn \ar[r] & \affr \ar[r] & \gln \ar[r] & 1}
\end{equation}

\begin{defn}
  \label{defn:iam}
An $n$-dimensional manifold $B$ is an \emph{integral affine manifold}
if it has an $\affr$-atlas - i.e. if there is an open cover
$\{U_{\alpha}\}$ of $B$ by coordinate charts so that the transition
functions lie in $\affr$.
\end{defn}

\begin{exm}
  \begin{enumerate}
  \item $\Rn$ is an integral affine manifold;
  \item any manifold $B$ admitting a local diffeomorphism $B \to \Rn$
    (see \cite{dui});
  \item $\tn$ inherits an integral affine structure from $\Rn$, since
    the action of $\zn$ is through integral affine transformations;
  \item more generally, if $B$ has an integral affine structure and we
    have a free and proper action of a group $\Gamma$ acting by
    integral affine transformations (i.e. we have a representation
    $\Gamma \to \mathrm{Aff}(B)$), then the quotient $B/\Gamma$ inherits an
    integral affine structure).
  \end{enumerate}
\end{exm}

It is well-known that the base space $B$ of a Lagrangian fibration is
an integral affine manifold (see \cite{dui}, \cite{bates_monchamp}, \cite{zun_symp2}). In
the following lemma we prove this fact and its converse, namely that
an integral affine manifold $B$ is the base space of some Lagrangian fibration.
\begin{lemma}\label{lemma:ias}
  A manifold $B$ is the base space of a Lagrangian fibration if and
  only if it is an integral affine manifold.
\end{lemma}
\begin{proof}
First we prove that the base space of a Lagrangian fibration naturally
inherits an integral affine structure. Suppose $B$ is the base of a
Lagrangian fibration. Choosing an open cover $\{U_i\}$ of $B$ by
coordinate neighbourhoods is equivalent to choosing local action-angle
coordinates $(a_i^1,\ldots,a_i^n,\alpha_i^1,\ldots,\alpha_i^n)$ over
each $U_i$. If needed, refine this cover so that it is a good cover in
the sense of Leray. We wish to determine how two sets of local action coordinates
are related. Let $(a^1,\ldots,a^n)$ and
$(\tilde{a}^1,\ldots,\tilde{a}^n)$ be local action coordinates
over $U_i \cap U_j$. Theorem \ref{thm:la} shows that each set induces
a system-preserving Hamiltonian $\tn$-action on $\pi^{-1}(U_i \cap
U_j)$. Taking  $(a^1,\ldots,a^n)$ as the local action coordinates, we
see that the only Hamiltonian vector fields that can induce a
system-preserving $S^1$-action on $\pi^{-1}(U_i \cap
U_j)$ are of the form

\begin{equation*} Y = \sum_{i=1}^n m_i X_i \end{equation*}

\noindent
where $X_i$ is the Hamiltonian vector field of the coordinate function
$a^i$ and $m_i \in \mathbb{Z}$. In particular, this shows that for
each $j$ there exist constants $m_{ij} \in \mathbb{Z}$ such that

\begin{equation*} \tilde{X}_j = \sum_{i=1}^n m_{ij} X_i
\end{equation*}

\noindent
where $\tilde{X}_j$ is the Hamiltonian vector field of the coordinate
function $\tilde{a}^j$ (here we use the fact that the overlap of any
two open sets is connected). Using the definition of Hamiltonian vector
fields we get that

\begin{align}
\mathrm{d}\tilde{a}^j & = \iota(\tilde{X}_j) \omega = \iota (
\sum_{i=1}^n m_{ij} X_i) \nonumber \\
 & = \sum_{i=1}^n m_{ij}(\iota(X_i)\omega) = \sum_{i=1}^n m_{ij}
\mathrm{d}a^i \nonumber \end{align}

\noindent
Hence we have that, for each $j$, the $\mathrm{d}\tilde{a}^j$ are
integral linear combinations of the $\mathrm{d}a^i$. This shows that
there exist constants $c_j \in \mathbb{R}$ such that

\begin{equation*} \tilde{a}^j =  \sum_{i=1}^n m_{ij} a^i + c_j
\end{equation*}
\noindent
Swapping the roles of $a^i$ and $\tilde{a}^j$, we get that, for each
$i$ there exist constants $\tilde{m}_{ji} \in \mathbb{Z}$ and $\tilde{c}_i \in
\mathbb{R}$ such that

\begin{equation*} a^i =  \sum_{j=1}^n \tilde{m}_{ji} \tilde{a}^j +
  \tilde{c}_i \end{equation*}
\noindent
Combining the two equations above, we see that the matrix $(m_{ij})$
must lie in $\gln$ and so the result follows.\\
\indent
We now need to show that the existence of an integral affine structure
on $B$ implies the existence of a Lagrangian fibration over $B$. Fix
an affine atlas on $B$ and let $\{U_{\alpha}\}$ be an open cover of
$B$ by sets in this
atlas. By shrinking the $U_{\alpha}$
if needed, we may assume that each $U_{\alpha}$ is contractible and
$U_\alpha \cap U_\beta$ is connected. Let
$\mathcal{A}_{\alpha \beta} = A_{\alpha \beta} + \mathbf{c}_{\alpha \beta}$
denote the transition function on the overlap $U_{\alpha} \cap
U_{\beta}$, where $A_{\alpha \beta} \in \gln$ and $\mathbf{c}_{\alpha \beta}
\in \Rn$.\\
\noindent
Define over each $U_{\alpha}$ the trivial Lagrangian fibration

\begin{equation*} \xymatrix@1{ \tn \; \ar@{^{(}->}[r] & U_{\alpha}
    \times \tn_{\alpha} \ar[r] & U_{\alpha}} \end{equation*}
\noindent
with symplectic form defined by

\begin{equation*} \omega_{\alpha} = \sum_i \mathrm{d}x^i_{\alpha}
  \wedge \mathrm{d}t^i_{\alpha} \end{equation*}

\noindent
where $x^i_{\alpha}$ are affine coordinates over $U_{\alpha}$ and
$t^i_{\alpha}$ are the standard coordinates on $\tn$. Define (wherever
it makes sense - i.e. on overlapping coordinate neighbourhoods) maps
$\psi_{\beta \alpha} : U_{\alpha} \cap U_{\beta} \times \tn_{\alpha}
\to U_{\alpha} \cap U_{\beta}
\times \tn_{\beta}$ by 

\begin{equation*} (\mathbf{x}_{\alpha},\mathbf{t}^i_{\alpha}) \mapsto
  (A_{\beta \alpha} \mathbf{x}_{\alpha} + \mathbf{c}_{\beta \alpha},
  (A_{\beta \alpha}^{-1})^T \mathbf{t}_{\alpha}) \end{equation*}

\noindent
Note that the maps $\{\psi_{\beta \alpha}\}$ satisfy the cocycle
condition, since the $\mathcal{A}_{\beta \alpha}$ clearly do and the
map from the affine group to its linear subgroup is a
homomorphism. Hence, these maps can be used to construct a
$\tn$-bundle over $B$. We note
that the $\omega_{\alpha}$ patch together to
give a globally defined symplectic form $\omega$ on the total space of
the bundle. Finally, we observe that each fibre is still Lagrangian
and so we have constructed a Lagrangian fibration over $B$ as
required.

\end{proof}

We note that using the methods of the above lemma, we can also show
that if $B$ is the base space of a Lagrangian fibration, the structure
group of $\cotan B$ can be reduced to (a subgroup of) $\gln$. The
converse, however, is \emph{not} true, as the following example shows.

\begin{exm}
$S^3$ is parallelisable and so the structure group of $\cotan S^3$ is
trivial. Suppose $S^3$ was the base space of a Lagrangian
fibration. Since it is $2$-connected (see \cite{zun_symp2}), the fibration
need be topologically trivial - i.e. $S^3 \times \mathbb{T}^3 \to
\mathbb{T}^3$. Call $\omega$ the symplectic form on this space. By the
K\"unneth formula, the inclusion of the fibre in the total space $
\iota: \mathbb{T}^3 \to S^3 \times \mathbb{T}^3 $ induces an
isomorphism

\begin{equation}
  \label{eq:4}
\iota^* : \mathrm{H}^2(S^3 \times \mathbb{T}^3; \R) \to \mathrm{H}^2(\mathbb{T}^3;\R)
\end{equation}
\noindent
Since $S^3 \times \mathbb{T}^3$ is closed, the cohomology class of
$\omega$ is non-zero. However, the fibre of the fibration is
Lagrangian and so $\iota^*[\omega] = 0$ and we have a
contradiction. 
\end{exm}

\noindent
This example generalises to the following result.

\begin{lemma}
  No sphere $S^n$ for $n \geq 2$ is the base space of a Lagrangian fibration.
\end{lemma}
\begin{proof}
  We only need to deal with the $n=2$ case. Since $S^2$ is
  simply-connected, if it is the base space of a Lagrangian fibration
  then the fibration has trivial monodromy. It follows from \cite{dui}
  that $S^2$ is parallelisable, but this is a contradiction.
\end{proof}

\begin{rk}
  
Lemma \ref{lemma:ias} implies that in order to determine whether a
manifold $B$ is the base space of a Lagrangian fibration, it is
sufficient to determine whether it is an integral affine manifold.
\end{rk}

Now we turn to the question of determining the structure group of a
Lagrangian fibration. Choose two sets of local action-angle
coordinates $(\mathbf{x}_{\alpha},\mathbf{t}_{\alpha})$,
$(\mathbf{x}_{\beta},\mathbf{t}_{\beta})$ on overlapping trivialising open
sets $U_{\alpha}$, $U_\beta$. It is well-known (see \cite{zun_symp2} and the
proof of Lemma \ref{lemma:ias}) that on $U_\alpha \cap U_\beta$ these
coordinates are related by the following
\begin{equation}
  \label{eq:3}
  (\mathbf{x}_\beta, \mathbf{t}_\beta) = (A_{\beta \alpha}
  \mathbf{x}_\alpha + \mathbf{c}_{\beta \alpha},(A^{-1}_{\beta
    \alpha})^T \mathbf{t}_\alpha + \mathbf{g}_{\beta \alpha}) 
\end{equation}

\noindent
where $A_{\beta \alpha} \in \gln$, $\mathbf{c}_{\beta \alpha} \in \Rn$
and $\mathbf{g}_{\beta \alpha}=(g^1_{\beta \alpha},\ldots, g^n_{\beta \alpha}) : U_\alpha \to \tn$ is a locally
defined function such that

\begin{displaymath}
  \sum_{i=1}^n g^i_{\beta \alpha} \de x^i_\alpha
\end{displaymath}
\noindent
is a closed form.\\
\begin{defn}
  \label{defn:afftor}
The \emph{group of affine toral automorphisms}, denoted by $\afft$, is
the semidirect product $ \gln \ltimes \tn$ with product defined by
\begin{displaymath}
  (A,x) \cdot (B,y) = (AB, Ay +x)
\end{displaymath}
\end{defn}
\noindent
We have a commutative diagram of short exact sequences

\begin{displaymath}
  \xymatrix{1 \ar[r] & \Rn \ar[r] \ar[d]^-{q} & \affr \ar[r] \ar[d] & \gln
    \ar[r] \ar@{=}[d]& 1 \\
1 \ar[r] & \tn \ar[r] & \afft \ar[r] & \gln \ar[r] & 1}
\end{displaymath}
\noindent
relating the integral affine group $\affr$ with the group of toral
affine automorphisms $\afft$, where the map $q : \Rn \to \tn$ is the standard
covering map.\\
Equation (\ref{eq:3}) shows that the structure group of a Lagrangian
fibration can be reduced to $\afft$. Hence, if we are given a
Lagrangian fibration $\fib$, by passing to the
associated principal $\afft$-bundle, we can study its isomorphism type
by looking at the homotopy type of the \emph{classifying map}

\begin{equation}
  \label{eq:6}
  f : B \to \mathrm{B} \afft 
\end{equation}
\noindent
where $\mathrm{B}\afft$ denotes the classifying space of the group $\afft$.

\section{Classifying space and obstructions to global action-angle
coordinates} \label{sec:cs}
As we have seen, the structure group of a Lagrangian fibration is
given by $\afft = \gln \ltimes \tn$. Let $\xymatrix@1{\afft \;
  \ar@{^{(}->}[r] & \bar{M} \ar[r] & B}$ denote the principal
$\afft$-bundle associated to a given Lagrangian bundle $\fib$ (for
details, see for instance \cite{davis_kirk}). The isomorphism type of
this bundle is determined by the homotopy class of a \emph{classifying
  map}

\begin{displaymath}
  f : B \to \caff
\end{displaymath}
\noindent
In this section we use this map to define the monodromy and Chern
class of a Lagrangian fibration.\\

We begin by identifying two important closed subgroups of $\afft$,
$H_1 = \toral$ and $H_2 = \linear$, called the \emph{toral} and
\emph{linear} subgroups respectively. For $i=1,2$, we have fibrations
\begin{displaymath}
  \xymatrix@1{\afft/H_i \; \ar@{^{(}->}[r] & \mathrm{B}H_i \ar[r] & \caff}
\end{displaymath}
\noindent
arising as follows. Let $\xymatrix@1{\afft \;\ar@{^{(}->}[r] &
  \mathrm{E}\afft \ar[r] & \caff}$ be the universal bundle for
$\afft$. Each $H_i$ acts freely on $\mathrm{E}\afft$ and so we can
take this space as a model for $\mathrm{E}H_i$. We note that $\afft /
(\toral) \cong \gln$ (as groups,since $\toral \lhd \afft$) and $\afft/(\linear) \cong \tn$ (as
topological spaces). 

\begin{lemma}\label{lemma:htygps}
 The homotopy groups of $\caff$ are given by
 \begin{equation*}
   \label{eq:8}
   \pi_i(\caff) =
   \begin{cases}
     \gln & \text{if $i=1$} \\
     \zn & \text{if $i=2$} \\
     0 & \text{otherwise}
   \end{cases}
 \end{equation*} 
\end{lemma}
\begin{proof}
 Consider the fibration arising from the
 action of the toral subgroup
 \begin{equation}
   \label{eq:7}
   \xymatrix@1{\gln \;\ar@{^{(}->}[r] & \ctoral \ar[r] & \caff}
 \end{equation}
 \noindent
 Note that $\ctoral$ is a $\mathrm{K}(\zn,2)$ as it is the classifying
 space of an $n$-torus. Using the long exact sequence in homotopy (see
 \cite{davis_kirk}), we obtain the result.
\end{proof}

We now deal with the other fibration, arising from the action of the
linear subgroup $\linear$

\begin{equation}
  \label{eq:9}
  \xymatrix@1{\tn \; \ar@{^{(}->}[r]^-{\iota} & \clinear \ar[r]^-{\pi} & \caff}
\end{equation}
\noindent
The long exact sequence in homotopy for this fibration ends as follows

\begin{equation}
  \label{eq:10}
  \xymatrix@1{\ldots \ar[r] & \pi_1(\tn) \ar[r]^-{\iota_*} & \pi_1(\clinear) \ar[r]^-{\pi_*} & \pi_1(\caff)
    \ar[r] & 0}
\end{equation}
\noindent
since $\pi_0 (\tn) = 0$. Note that $\pi_1(\clinear)$ is a
$\mathrm{K}(\gln,1)$ since $\gln$ is discrete group. In particular, we
have that $\pi_1(\clinear) \cong \pi_1(\caff)$. However, we also have
that $\pi_*$ is a surjection and so

\begin{equation}
  \label{eq:11}
  \pi_1(\clinear) \cong \pi_1(\caff) \cong \pi_1(\clinear)/\ker \pi_* 
\end{equation}
\noindent
where the first isomorphism follows from Lemma \ref{lemma:htygps}.
Equation (\ref{eq:11}) shows that $\pi_*$ is actually an isomorphism
and we will use this fact when defining the Chern class of a
Lagrangian fibration.\\

The key idea that we use to define these invariants is to define them as obstructions to the existence of a lift
of the classifying map $f$ to maps $\bar{f} : B \to \ctoral$ and
$\tilde{f} : B \to \clinear$ into the total spaces of the fibrations
constructed above. We will use the following theorem, quoted here
without proof (see \cite{huse}).

\begin{thm}\label{thm:huse}
  Let $H \leq G$ be a closed subgroup of a topological group $G$. The
  structure group of a principal $G$-bundle $\xymatrix@1{G
   \; \ar@{^{(}->}[r] & P \ar[r] & B}$ can be reduced to $H$ if and only
  if there exists a lift $ \bar{F}: B \to \mathrm{B}H $ of the
  classifying map $F: B \to \mathrm{B}G$, as illustrated below

  \begin{displaymath}
       \xymatrix{ & \mathrm{B}H \ar[d]^-{\pi_\mathrm{B}} \\
   B \ar@{.>}[ur]^-{\bar{F}} \ar[r]^-{F} & \mathrm{B}G}
  \end{displaymath}
  where $\xymatrix@1{G/H \;\ar@{^{(}->}[r] & \mathrm{B} H \ar[r]& \mathrm{B}
    G}$ is a fibration obtained as above.
\end{thm}

\subsection{Monodromy} \label{sec:monodromy}
We first recall the classical definition (see \cite{dui},
\cite{bates_monchamp}, \cite{cush_dav} and \cite{zun_symp2}) of
monodromy and then give an entirely topological definition. Associated
to a Lagrangian fibration $\fib$, there is a \emph{period lattice
  bundle}
\begin{displaymath}
  \xymatrix@1{\mathrm{H}_1(F,\mathbb{Z}) \;\ar@{^{(}->}[r] & \mathcal{P}
    \ar[r] & B}
\end{displaymath}
\noindent
where the generators of $\mathrm{H}_1(F,\mathbb{Z})$ define a local
Hamiltonian $\tn$-action on $M$ (see the proof of lemma
\ref{lemma:ias} or \cite{dui}). An explicit construction of this period lattice bundle is explained in \cite{cush_dav}. We want to remark that if we fix an good (in the sense of Leray) open cover $\{U_\alpha\}$ of $B$ by trivialising neighbourhoods for $\fib$, then we can locally trivialise the period lattice bundle over this cover, with transition functions over $U_\alpha \cap U_\beta$ given by

\begin{displaymath}
(\mathbf{x}_\beta, \mathbf{z}_\beta) = (A_{\beta \alpha} \mathbf{x}_\alpha + \mathbf{c}_{\beta \alpha}, (A_{\beta \alpha}^{-1})^T \mathbf{z}_\alpha)
\end{displaymath}
\noindent 
where $\mathrm{z} \in \mathrm{H}_1(F_x, \Z)$. In particular, it is important to notice that the trivialisations and transition functions for the period lattice bundle are completely determined by those for the associated Lagrangian fibration. \\

The classical definition of monodromy
is precisely as an obstruction for these local Hamiltonian
$\tn$-actions to patch together to yield a global $\tn$-action. If the
monodromy vanishes, then this $\tn$-action is actually free (this is
most evident in the proof of the Liouville-Arnold theorem (Theorem \ref{thm:la})
or in \cite{cush_dav}) and so the Lagrangian fibration is a principal
$\tn$-bundle. Conversely, if there exists a globally defined free
Hamiltonian $\tn$-action on the total space $M$ of a Lagrangian
fibration, then the fibration is a principal bundle and the period
lattice bundle is trivial. It follows therefore from the classical
definition of monodromy (see \cite{cush_dav}) that the monodromy of
the Lagrangian fibration is trivial. In particular, this shows that
the monodromy can be seen as the obstruction to reducing the structure
group of the fibration to the toral subgroup $\toral$. \\

We have the fibration

 \begin{displaymath}
  \xymatrix@1{\gln \;\ar@{^{(}->}[r] & \ctoral \ar[r] & \caff}
\end{displaymath}
\noindent
where the fibre is discrete and the total space is simply connected. A
lift $\bar{f}: B \to \ctoral$ of the classifying map

\begin{displaymath}
  \xymatrix{ & \ctoral \ar[d] \\
B \ar@{.>}[ur]^-{\bar{f}} \ar[r]^-{f} & \caff}
\end{displaymath}
\noindent
exists if and only if 
\begin{equation}
  \label{eq:12}
  f_* : \pi_1(B) \to \pi_1(\caff) \cong \gln
\end{equation}
\noindent
is trivial.\\

It is tempting to define $f_*$ as the monodromy associated to
the Lagrangian fibration, but this would be a naive definition since
we would like monodromy to be an invariant of the \emph{isomorphism}
type of the fibration. Note that if two Lagrangian fibrations $\fib$
and $\fibr$ over $B$ are isomorphic, then clearly the associated
period lattice bundles $\fibpl$ and $\fibrpl$ are isomorphic.
It is
well-known (see \cite{cush_dav}) that the structure group of these
discrete bundles is $\gln$ and, hence, their isomorphism type is
determined by the homotopy class of a map
\begin{displaymath}
  \phi : B \to \cgln
\end{displaymath}
Since $\cgln$ is a $\mathrm{K}(\gln,1)$, then we have that such
homotopy classes are parametrised by
\begin{displaymath}
  [B,\cgln] = \mathrm{Hom}(\pi_1(B),\gln)/(\alpha \sim \beta \alpha
  \beta^{-1})
\end{displaymath}
\noindent
Hence if we let $f, f'$ be the classifying maps associated to the
Lagrangian fibrations $\fib$, $\fibr$ respectively, we have that
$f_*(\pi_1(B))$ and $f'_*(\pi_1(B))$ are \emph{conjugate} subgroups of
$\gln$.

\begin{defn} \label{defn:monodromy-1}
Let $f: B \to \caff$ be the classifying map associated to a Lagrangian
fibration and let $[f_*]$ denote the equivalence class of subgroups of
$\gln$ that are conjugate to $f_*(\pi_1(B))$. Then we call $[f_*]$ the
\emph{free monodromy} of the Lagrangian fibration.  
\end{defn}

From now on, whenever we mention the monodromy of a Lagrangian
fibration we will mean it in the sense of Definition
\ref{defn:monodromy-1}, although we abuse notation and still denote it
as a map $f_*: \pi_1(B) \to \pi_1(\caff)$.

\subsection{Chern Class}\label{sec:chern-class}
Fix a Lagrangian fibration $\fib$ with monodromy $f_* : \pi_1(B) \to
\pi_1(\caff)$. Using the methods of Lemma \ref{lemma:ias}, we can
construct the \emph{trivial $f_*$-fibration} as follows. Let
$\{U_\alpha\}$ be an open cover of $B$ by trivialising open sets for
$\fib$. We know from equation (\ref{eq:3}) that the transition
functions on $U_\alpha \cap U_\beta$ are of the form 

\begin{displaymath}
(\mathbf{x}_\beta, \mathbf{t}_\beta) = (A_{\beta \alpha} \mathbf{x}_\alpha + \mathbf{c}_{\beta \alpha}, (A_{\beta \alpha}^{-1})^T \mathbf{t}_\alpha + \mathbf{g}_{\beta \alpha})
\end{displaymath}
\noindent
Define a new fibration with same trivialisations as the original one and transition functions given by

\begin{displaymath}
(\mathbf{x}_\beta, \mathbf{t}_\beta) = (A_{\beta \alpha} \mathbf{x}_\alpha + \mathbf{c}_{\beta \alpha}, (A_{\beta \alpha}^{-1})^T \mathbf{t}_\alpha)
\end{displaymath}

This new fibration is still Lagrangian since the local symplectic
structures patch together and the fibres are Lagrangian. Note that its
structure group lies entirely in the linear subgroup $\linear$ of
$\afft$. This new Lagrangian fibration has the same monodromy as the
original fibration, since they have isomorphic period lattice bundles,
as can be seen by the construction of a period lattice bundle in
\cite{cush_dav}. We can use the trivialising cover
$\{U_\alpha\}$ of $B$ to define local trivialisations for $\cotan B$,
given by 
 
\begin{displaymath}
(\mathbf{x}_\beta, \mathbf{y}_\beta) = (A_{\beta \alpha} \mathbf{x}_\alpha + \mathbf{c}_{\beta \alpha}, (A_{\beta \alpha}^{-1})^T \mathbf{y}_\alpha)
\end{displaymath}
\noindent
where $\mathbf{y} \in \cotan_{x} B$. We have also seen that the
transition functions for the period lattice bundle $\mathcal{P} \to B$
are precisely the same and so we deduce that the trivial
$f_*$-fibration is isomorphic to the Lagrangian bundle $\cotan B /
\mathcal{P} \to B$ (since they have
the same transition functions and trivialisations over a common
trivialising cover). This explains why
we call this fibration the trivial $f_*$-fibration, as it has trivial
Chern class in the classical definition (see \cite{dui}). Indeed,
Duistermaat defined the the Chern class of a Lagrangian fibration to
be the obstruction to the existence of a bundle isomorphism 
\begin{displaymath}
  \xymatrix{ M \ar[d] \ar[r]^-{\Theta} & \cotan B/\mathcal{P} \ar[d]
    \\
B \ar[r]^-{\theta} & B}
\end{displaymath}
We note that the structure group of $\cotan B /\mathcal{P} \to B$ can
be reduced to the linear subgroup $\linear$ of $\gln$. Hence the
structure group of a Lagrangian fibration that is isomorphic to the
trivial $f_*$-fibration can be reduced to $\linear$. Conversely,
suppose that the structure group of a Lagrangian fibration $\fib$
reduces to $\linear$. Then Theorem \ref{thm:huse} shows that we have a
lift $g: B \to \clinear$ as shown below

\begin{displaymath}
  \xymatrix{ & \clinear \ar[d]^-{\pi} \\
    B \ar[ur]^-{g} \ar[r]^-{f} & \caff}
\end{displaymath}

We wish to show that this Lagrangian fibration is isomorphic to the
trivial $f_*$-fibration. To this end, let $f_0 : B \to \caff$ denote
the (homotopy type of the) classifying map associated to the trivial
$f_*$-fibration. Equation (\ref{eq:11}) shows that the map $\pi_* :
\pi_1(\clinear) \to \pi_1(\caff)$ is an isomorphism. Since the
structure group of the trivial $f_*$-fibration can be reduced to
$\linear$, then $f_0$ has a lift $G : B \to \clinear$. Since $\pi_*$
is an isomorphism, it follows that $g_* = G_*$ (up to
conjugation). The isomorphism type of maps from $B$ to $\clinear$ is
classified precisely by $\mathrm{Hom}(\pi_1(B),\gln)$ up to
conjugation and so $g$ is homotopic to $G$. It is therefore
straightforward to see that $f$ is homotopic to $f_0$ and so the
original Lagrangian fibration is isomorphic to the trivial
$f_*$-fibration. \\

From here on, until the end of the section, we assume that $B$ is a
CW-complex. The obstruction to finding a lift $g : B \to \clinear$ of
the classifying map $f$ of a Lagrangian fibration is an element of a
twisted cohomology group (see \cite{davis_kirk}
\begin{equation}
  \label{eq:5}
  c \in \mathrm{H}^2(B, \zn_\rho)
\end{equation}
\noindent
where the map $\rho : \pi_1(B) \to \zn$ is the composite
\begin{displaymath}
  \xymatrix@1{ \pi_1(B) \ar[r]^-{f_*} & \pi_1(\caff) \ar[r]^-{\chi} & \mathrm{Aut}(\pi_1(\tn)) }
\end{displaymath}
\noindent
Here $\chi$ denotes the action of the fundamental group of $\caff$ on
the fundamental group of the fibre $\tn$ (see \cite{davis_kirk}). We
will abuse notation and write $f_*$ also to mean $\chi \circ f_*$.

\begin{defn} \label{defn:cc}
  $c$ is the \emph{Chern class} associated to the (isomorphism type of
  a) Lagrangian fibration determined by a classifying map $f: B \to \caff$.
  
\end{defn}


\section{Examples of twisted Lagrangian fibrations}\label{sec:exs}
In this section we explicitly construct examples of regular Lagrangian
fibrations with non-trivial monodromy and non-trivial Chern class,
following ideas of Bates (see \cite{bates_ob}) and Sakamoto and
Fukuhara (see \cite{sf}). First we show how to construct a Lagrangian
fibration with base space $\rpr$ and calculate its monodromy $f_*$,
which is non-trivial since $\rpr$ is not parallelisable (see
\cite{dui}). Then, using a result of Dazord and Delzant (see
\cite{daz_delz}) adapted by Zung in \cite{zun_symp2}, we deduce that
\emph{all} cohomology classes in $\twcoho$ can be realised as the
Chern class of some Lagrangian fibration with base $\rpr$. Finally,
we construct these explicitly, first just topologically following
ideas of \cite{sf} and then also taking care of the symplectic form,
generalising a construction carried out in \cite{bates_ob}.\\

For notational ease, set $X=\rpr$ and let $\tilde{X} =\rt - \{0\}$
denote its universal cover. $\tilde{X}$ is the base space of a trivial
Lagrangian fibration

\begin{equation}
  \label{eq:13}
  \xymatrix@1{\tth \; \ar@{^{(}->}[r] & \tilde{X} \times \tth
    \ar[r]^-{\tilde{\pi}} & \tilde{X}}
\end{equation}
\noindent
We remark that we are using the standard integral affine structure on
$\tilde{X}$. There is a free and proper $\ztw$-action on $\tilde{X}$,
given by
\begin{equation}
  \label{eq:14}
  a \cdot \xvec = - \xvec
\end{equation}
\noindent
where $a \in \ztw$ is the non-trivial element. We can lift this action
to the total space of (\ref{eq:13}) by setting

\begin{equation}
  \label{eq:15}
  a \cdot (\xvec, \tvec) = (-\xvec, -\tvec)
\end{equation}

This defines a free and proper $\ztw$-action on $\tilde{X} \times \tth$,
which is by bundle automorphisms (the bundle is trivial) and
symplectomorphisms (since the symplectic form $\omega$ is just $ \sum_{i=1}^3
\de x^i \wedge \de t^i$). Hence, if we quotient the bundle
(\ref{eq:13}) by this action, we obtain a Lagrangian fibration

\begin{equation}
  \label{eq:16}
  \xymatrix@1{\tth \; \ar@{^{(}->}[r] & M_0  \ar[r] & X}
\end{equation}

Since $X$ is not parallelisable (let alone orientable), this fibration
need have non-trivial monodromy.

\begin{claim}
  The monodromy of this fibration is given by the map
  \begin{align*}
    f_*: \pi_1(\rpr) & \to \mathrm{GL}(3,\Z) \\
a & \mapsto - I 
  \end{align*}
\end{claim}
\begin{proof}
  We work on the double covering
  \begin{equation}
    \xymatrix@1{\tth \; \ar@{^{(}->}[r] & \tilde{X} \times \tth
    \ar[r]^-{\tilde{\pi}} & \tilde{X}}
  \end{equation}
Let $c: [0,1] \to \tilde{X}$ be a lift of the generator $a$ of the
fundamental group of $\rpr$ - take for instance the path $c(t) =
(\cos(\pi t),\sin(\pi t), 0)$. It is well-known (see
\cite{bates_monchamp}) that an integral affine manifold $N$ admits a
flat, torsion-free connection $\nabla$, whose holonomy equals the
monodromy of the Lagrangian fibration over $N$ constructed in Lemma
\ref{lemma:ias}. Let $\nabla$ denote the connection on $X$ arising
from the Lagrangian fibration we have just constructed and let
$\tilde{\nabla}$ denote the pullback of this connection by the
covering map $q: \tilde{X} \to X$. By construction, $\tilde{\nabla}$
is the standard connection on $\tilde{X}$ - here we recall that we
started with the standard integral affine structure on $\tilde{X}$ to
construct the Lagrangian fibration with base $X$. We wish to calculate
the holonomy of $\nabla$. Consider any non-zero $\mathbf{v} \in
\cotan_{(1,0,0)} \tilde{X}$ and parallel transport it along
$c$. Recall that the cotangent bundle $\cotan \tilde{X}$ is trivial, so that we
can canonically identify any two of the fibres via the standard
connection $\tilde{\nabla}$. Since parallel transport is trivial,
$\mathbf{v}$ maps to $\mathbf{v} \in \cotan_{(-1,0,0)}
\tilde{X}$. However, under the lift of the $\ztw$-action on
$\tilde{X}$, we have the following identification
\begin{displaymath}
  ((1,0,0),\mathbf{v}) \sim ((-1,0,0),-\mathbf{v})
\end{displaymath}
Hence, projecting down to $X$, we get that the holonomy of the
connection $\nabla$ is given by
\begin{displaymath}
  a \mapsto -I
\end{displaymath}
as required.
\end{proof}

In section \ref{sec:chern-class} we showed that the Chern class of
Lagrangian fibrations with base $X$ and monodromy given by $f_*$ are
classified by elements of the twisted cohomology group $\twcoho$. A
natural question to ask is \emph{which} twisted cohomology classes in
$\twcoho$ can be realised as the Chern class of a Lagrangian fibration
over $X$ with the given monodromy. In the general case, we have no
concrete answer, although in this case we can use a result of Dazord
and Delzant (see \cite{daz_delz}) and Zung (see \cite{zun_symp2}) to
deduce that \emph{all} twisted cohomology classes can be realised.

\begin{thm}[Dazord-Delzant, Zung] \label{thm:daz}
  Let $Y$ be an $n$-dimensional manifold that is the base space of a
  Lagrangian fibration with monodromy given by $f_*$. If
  $\mathrm{H}^3(Y;\R) = 0$, then \emph{all} cohomology classes in
  $\mathrm{H}^2(Y;\zn_{f_*})$ can be realised as the Chern class of
  some Lagrangian fibration with base $Y$ and monodromy as
  before. Furthermore, if $\mathrm{H}^2(Y;\R) = 0$, then the
  Lagrangian class vanishes for any Chern class.
\end{thm}

Since $\rpr$ satisfies the assumptions of Theorem \ref{thm:daz}, we
deduce that all elements of $\twcoho$ arise as the Chern class of some
Lagrangian fibration with base $\rpr$ and that the symplectic
structure on the total space of these fibrations preserving the
Lagrangian structure is unique up to symplectomorphism.\\
We now compute $\twcoho$ and show that
\begin{displaymath}
  \twcoho \cong \zt
\end{displaymath}
\noindent
using standard techniques in algebraic topology (see
\cite{davis_kirk}). Since $\rpr$ is homotopy equivalent to $\rpt$, it
will suffice to show that $\twco \cong \zt$. Recall that $\rpt$ has a
standard cell decomposition
\begin{displaymath}
  \rpt = e^0 \cup e^1 \cup e^2
\end{displaymath}
inducing a cell decomposition on its universal cover
\begin{displaymath}
  S^2 = e^0_+ \cup e^0_- \cup e^1_+ \cup e^1_- \cup e^2_+ \cup e^2_- 
\end{displaymath}
Set $\pi_1(\rpt) = \pi$. Denote by $\mathrm{C}_i(S^2)$ the
$\Z[\pi]$-module of $i$-cells in $S^2$. For $i=0,1,2$
$\mathrm{C}_i(S^2)$ is a $1$-dimensional $\Z[\pi]$-module, generated
by $e^i_+$. We denote the action of $\pi$ on $\mathrm{C}_i(S^2)$ by
multiplication by $t$, following the notation in
\cite{davis_kirk}. The twisted cellular chain complex for $\rpt$ is
given by

\begin{displaymath}
  \xymatrix@1{0 \ar[r] & \mathrm{C}_2(S^2) \ar[r]^-{\partial_2} &
    \mathrm{C}_1(S^2) \ar[r]^-{\partial_1} & \mathrm{C}_0(S^2) \ar[r]
    & 0}
\end{displaymath}
\noindent
where the differentials are defined as follows
\begin{align*}
  \partial_2(e^2_+) &= (1+t)e^1_+ \\
  \partial_1(e^1_-)&= (1-t)e^0_+
\end{align*}
\noindent
The associated twisted cochain complex is given by

\begin{displaymath}
  \xymatrix@1{0 \ar[r] & \mathrm{Hom}_{\Z[\pi]}(\mathrm{C}_0(S^2),\zt) \ar[r]^-{\delta_1} &
    \mathrm{Hom}_{\Z[\pi]}(\mathrm{C}_1(S^2),\zt) \ar[r]^-{\delta_2} & \mathrm{Hom}_{\Z[\pi]}(\mathrm{C}_2(S^2),\zt) \ar[r]
    & 0}
\end{displaymath}
\noindent
where $\delta_1,\delta_2$ are adjoint maps to $\partial_1, \partial_2$
respectively and we think of $\zt$ as a $\Z[\pi]$-module via the
representation $f_* : \pi \to \mathrm{GL}(3;\Z)$ (extended naturally
to the whole ring).\\
\indent
We are interested in
\begin{displaymath}
  \twco = \frac{\mathrm{Hom}_{\Z[\pi]}(\mathrm{C}_2(S^2),\zt)}{\mathrm{im}\,\delta_2}
\end{displaymath}
Note that $\mathrm{Hom}_{\Z[\pi]}(\mathrm{C}_2(S^2),\zt) \cong \zt$,
since any $\Z[\pi]$-module homomorphism $\chi : \mathrm{C}_2(S^2) \to
\zt$ is completely determined by its value on the generator $e^2_+$ of
$\mathrm{C}_2(S^2)$. If $\tau \in
\mathrm{Hom}_{\Z[\pi]}(\mathrm{C}_1(S^2),\zt)$, then
\begin{align*}
  (\delta_2 \tau)(e^2_+)& = \tau (\partial_2 e^2_+)  =
  \tau((1+t)e^1_+)\\
& = \tau(e^1_+) + f_*(t) \tau(e^1_+) = \tau(e^1_+) - \tau(e^1_+)=0
\end{align*}
and so $\mathrm{im}\,\delta_2 = \{0\}$. Hence, $\twco \cong \zt$ as
claimed. \\

This is  the space of Chern classes of Lagrangian fibrations with
base $\rpr$ and monodromy given by $f_*$. We can now construct these
fibrations explicitly, first topologically (using ideas from \cite{sf}), and then showing how to
construct the symplectic form. Fix an element $(m,n,p) \in \zt$. Think
of $\rpr$ as $D^2/\sim \times \R$, where $D^2$ denotes the closed disk in
two dimensions and the equivalence relation is defined by $ x \sim -x
$ if $|x|=1$. Let $B(\epsilon)$ be a closed disk of radius $\epsilon >
0$ about $0 \in D^2$ and set $B = B(\epsilon) \times \R$. Recall that
we have a Lagrangian bundle
\begin{displaymath}
   \xymatrix@1{\tth \;\ar@{^{(}->}[r] & M_0  \ar[r] & X}
\end{displaymath}
\noindent
whose Chern class is clearly $(0,0,0) \in \zt$. Construct a new bundle
$M \to \rpr$ by defining the total space by
\begin{displaymath}
  M = (M_0 - \pi^{-1}(\mathrm{int}B)) \cup (\tth \times B)
\end{displaymath}
where $\tth \times \partial B$ is attached to $\pi^{-1}(\partial B)$
via the map
\begin{displaymath}
  h ((\epsilon \cos (\theta),\epsilon \sin(\theta), s), \tvec) =
  ((\epsilon \cos (\theta),\epsilon \sin(\theta), s),
  \frac{\theta}{2\pi}(m,n,p) + \tvec)
\end{displaymath}
(the fibre is thought of as $\tth = \rt/\zt$). This defines a
$\tth$-bundle with base space $\rpr$, the monodromy is
unchanged and its Chern class (seen in this case as the obstruction
to the existence of a global section) is $(m,n,p) \in \zt$. Theorem
\ref{thm:daz} ensures that there exists a unique (up to
symplectomorphism) symplectic structure on $M$ that makes it
Lagrangian. Since the choice of $(m,n,p)$ was arbitrary, we have
constructed \emph{all} the symplectic types of Lagrangian fibrations
with base space $\rpr$ and monodromy given by $f_*$.\\

These bundles can also be constructed in a different
fashion, which elucidates how the symplectic structure comes
about. This method generalises ideas from Bates' paper
\cite{bates_ob}. The universal cover $\tilde{X}$ of $\rpr$ is
diffeomorphic to $S^2 \times \R$. Consider the trivial Lagrangian
fibration
\begin{displaymath}
  \xymatrix@1{\tth \;\ar@{^{(}->}[r] & \tilde{X} \times \tth
    \ar[r]^-{\tilde{\pi}} & \tilde{X}}
\end{displaymath}
\noindent
with the standard symplectic form on the total space. Fix some small
$\epsilon > 0$. Let $B_{\pm}(\epsilon) \subset S^2$ be closed discs of
radius $\epsilon$ centred at $(1,0,0)$ and $(-1,0,0)$ respectively
(here we identify the open upper and lower hemispheres of $S^2$ with
$\R^2$). Set $B_{\pm} = B_{\pm}(\epsilon) \times \R \,(=\bigcup_{r >0} B_{\pm}(r
\epsilon) \times \{r\})$. We define a new bundle with base space
$\tilde{X}$ by constructing the following total space
\begin{displaymath}
  \tilde{M} = ((\tilde{X} \times \tth) -
  (\tilde{\pi}^{-1}(\mathrm{int}B_+) \cup
  \tilde{\pi}^{-1}(\mathrm{int}B_-))) \cup (\tth \times B_{+}) \cup
  (\tth \times B_{-})
\end{displaymath}
\noindent
where $\tth \times \partial B_+$ is attached to
$\tilde{\pi}^{-1}(\partial B_+)$ via the attaching map

\begin{displaymath}
h_+(\xvec,\tvec) = (\xvec, t^1 + \arg(x^1 + \imath x^2), t^2 + \frac{1}{2}\log(\frac{(x^1)^2+(x^2)^2}{1+(x^2)^2}), t^3)  
\end{displaymath}
\noindent
and  $\tth \times \partial B_-$ is attached to $\tilde{\pi}^{-1}(\partial B_-)$ via the attaching map

\begin{displaymath}
h_-(\xvec,\tvec) = (\xvec, t^1 - \arg(x^1 + \imath x^2) + \pi, t^2 - \frac{1}{2}\log(\frac{(x^1)^2+(x^2)^2}{1+(x^2)^2}), t^3)  
\end{displaymath}
\noindent
Note that here we think of $\tth = \rt/2\pi \zt$ and so there is no
ambiguity when defining the argument function $\arg$. This defines a
$\tth$-bundle with base space $\tilde{X}$. Moreover, the local
(natural) symplectic forms patch together (since $\partial_2
(\arg(x^1+\imath x^2)) = \partial_1
(\frac{1}{2}\log(\frac{(x^1)^2+(x^2)^2}{1+(x^2)^2}))$) to yield a
global symplectic form $\tilde{\omega}$ with respect to which the
fibres are Lagrangian. Hence this new bundle is a Lagrangian
fibration. \\
\noindent
The standard antipodal action on $\tilde{X} \times \tth$ sends $(\tilde{X} \times \tth) -
(\tilde{\pi}^{-1}(\mathrm{int}B_+) \cup
\tilde{\pi}_1^{-1}(\mathrm{int}B_-)))$ to itself and swaps  $(\tth \times B_{+})$ and $
(\tth \times B_{-})$. Note further that the action is compatible
with the attaching maps, since

\begin{align*}
a \cdot h_+(\xvec,\tvec) &= a \cdot (\xvec, t^1 + \arg(x^1 + \imath x^2), t^2 + \frac{1}{2}\log(\frac{(x^1)^2+(x^2)^2}{1+(x^2)^2}), t^3)  \\
&= (-\xvec,- t^1 - \arg(x^1 + \imath x^2),- t^2 - \frac{1}{2}\log(\frac{(x^1)^2+(x^2)^2}{1+(x^2)^2}), -t^3)
\end{align*}
\noindent
and

\begin{align*}
h_- (a \cdot(\xvec,\tvec)) &= h_- (-\xvec, -\tvec) \\
&= (-\xvec,- t^1 - \arg(-x^1 - \imath x^2)+ \pi,- t^2 - \frac{1}{2}\log(\frac{(-x^1)^2+(-x^2)^2}{1+(-x^2)^2}), -t^3) \\
&= (-\xvec,- t^1 - \arg(x^1 + \imath x^2),- t^2 - \frac{1}{2}\log(\frac{(x^1)^2+(x^2)^2}{1+(x^2)^2}), -t^3)
\end{align*}
\noindent
since $\arg(-x^1-\imath x^2) = \arg(x^1 + \imath x^2) + \pi$ (mod
$2\pi$). Furthermore, the same proof shows that it is a $\Z_2$-action
on the total space of the new Lagrangian fibration. In particular, the
above also shows that this $\Z_2$-action is by bundle automorphisms
(as it commutes with the attaching maps) and by
symplectomorphisms. Taking quotients, we obtain a Lagrangian fibration

\begin{displaymath}
\xymatrix@1{ \tth \;\ar@{^{(}->}[r] & M(1,0,0) \ar[r] & \rpr}
\end{displaymath}
\noindent
Its monodromy is again given by $f_*$ and, by construction, its Chern
class is precisely $(1,0,0)$. \\

We now show how to generalise this construction to produce a
Lagrangian fibration with base $\rpr$ and
arbitrary Chern class. Let $G \in \mathrm{GL}(3,\Z)$ and define
disjoint cones $B^{G}_{\pm}$ so that $B_{\pm} = G
B^G_{\pm}$. Set $\phi_+ (\xvec) = (\arg(x^1 + \imath x^2),
\frac{1}{2}\log(\frac{(x^1)^2+(x^2)^2}{1+(x^2)^2}), 0)$
and $\phi_- (\xvec) = (-\arg(x^1 + \imath x^2)+\pi,
-\frac{1}{2}\log(\frac{(x^1)^2+(x^2)^2}{1+(x^2)^2}), 0)$.
We define a new bundle with base space $\tilde{X}$ whose
base space is given by
\begin{displaymath}
  \tilde{M}^G = ((\tilde{X} \times \tth) -
  (\tilde{\pi}^{-1}(\mathrm{int}B^{G}_+) \cup
  \tilde{\pi}^{-1}(\mathrm{int}B^{G}_-))) \cup (\tth \times B^G_{+}) \cup
  (\tth \times B^G_{-})
\end{displaymath}
\noindent
where $\tth \times \partial B^G_+$ is attached to
$\tilde{\pi}^{-1}(\partial B^G_+)$ via the attaching map

\begin{displaymath}
h^G_+(\xvec,\tvec) = (\xvec, \tvec + (G^{-1})^T \phi_+ (G\xvec))  
\end{displaymath}
\noindent
and  $\tth \times \partial B^G_-$ is attached to $\tilde{\pi}^{-1}(\partial B^G_-)$ via the attaching map

\begin{displaymath}
h^G_-(\xvec,\tvec) = (\xvec, \tvec + (G^{-1})^T \phi_-(G \xvec))  
\end{displaymath}
\noindent
Since $G$ is linear, we can still define a $\Z_2$-action as we did in
the previous example and all the claims made above follow through as
well (again, by linearity of $G$). Set $(m,n,p) = (G^{-1})^T(1,0,0)$. 
In particular, by passing to the quotient, we obtain a
Lagrangian fibration

\begin{displaymath}
\xymatrix@1{ \tth \;\ar@{^{(}->}[r] & M(m,n,p) \;\ar[r] & \rpr}
\end{displaymath}
\noindent
whose monodromy is still given by $f_*$ and whose Chern class is
precisely given by $(m,n,p)$. Since $G \in
\mathrm{GL}(3,\Z)$ was arbitrary and the orbit
of $(1,0,0)$ under the natural action of $\mathrm{GL}(3,\Z)$ on $\zt$
is $\zt$, we can construct in this fashion all
the topological (and symplectic) types of Lagrangian fibrations with
base space $\rpr$ and monodromy given by $f_*$. It is important to
notice that we exploit in a crucial fashion the natural linear action of
$\mathrm{GL}(3,\Z)$ on $\tth$ in this example.

\section{Concluding Remarks}\label{sec:concl}
While the construction carried in section \ref{sec:exs} might be a
simple exercise in differential topology, we believe it has
nonetheless highlighted the importance of the integral affine
structure on the base space of a Lagrangian fibration. This is also
why the whole paper has been centred on understanding topological
invariants of Lagrangian fibrations primarily using this rigid
structure on the base. Furthermore, this remark might be of some use
in understanding special Lagrangian fibrations, which play a
fundamental role in
mirror symmetry.


\bibliographystyle{abbrv}
\bibliography{mybib}
\end{document}